\documentclass[11pt]{article}

\usepackage[margin=1in]{geometry}
\usepackage[T1]{fontenc}
\usepackage[utf8]{inputenc}
\usepackage{lmodern}
\usepackage{microtype}
\usepackage{amsmath,amssymb,amsthm,mathtools,bm,mathrsfs}
\usepackage{booktabs,array,longtable,tabularx}
\usepackage{enumitem}
\usepackage{xcolor}
\usepackage{hyperref}
\usepackage[nameinlink,capitalize]{cleveref}
\usepackage[numbers,sort&compress]{natbib}

\hypersetup{
 hidelinks,
 pdftitle={Matching Lower Bounds for Randomized First-Order Methods in Hessian-Lipschitz Nonconvex Optimization},
 pdfauthor={Anonymous Authors}
}
\newtheorem{theorem}{Theorem}[section]
\newtheorem{proposition}{Proposition}[section]
\newtheorem{lemma}{Lemma}[section]
\newtheorem{corollary}{Corollary}[section]

\theoremstyle{definition}
\newtheorem{definition}{Definition}[section]
\newtheorem{construction}{Construction}[section]

\theoremstyle{remark}

\newcommand{\R}{\mathbb{R}}

\newcommand{\Prob}{\mathbb{P}}
\newcommand{\eps}{\varepsilon}
\newcommand{\Lip}{\operatorname{Lip}}
\newcommand{\spanop}{\operatorname{span}}
\newcommand{\diag}{\operatorname{diag}}

\newcommand{\FO}{\mathrm{FO}}
\newcommand{\rand}{\mathrm{rand}}

\newcommand{\cF}{\mathcal{F}}

\newcommand{\cK}{\mathcal{K}}
\newcommand{\cV}{\mathcal{V}}

\newcommand{\cL}{\mathcal{L}}
\newcommand{\Q}{\mathsf{Q}}
\newcommand{\op}{\mathrm{op}}
\newcommand{\1}{\mathbf{1}}
\newcommand{\good}{\mathsf{Good}}

\title{\textbf{Matching Lower Bounds for Randomized First-Order Methods in Hessian-Lipschitz Nonconvex Optimization}}
\newcommand{\authornames}{Haihan Zhang, Wendao Wu, Chenheng Zhang, Yanyi Li, Chunyuan Zheng, Cong Fang, Haoxuan Li, Zhouchen Lin}
\newcommand{\affiliation}{Peking University}
\newcommand{\emailA}{colombo.2026@outlook.com}

\newcommand{\mail}[1]{\href{mailto:#1}{\texttt{#1}}}
\newif\ifanonymous
\anonymousfalse
\ifanonymous
\author{Anonymous Authors}
\hypersetup{pdfauthor={Anonymous Authors}}
\else
\author{Haihan Zhang$^{1,*}$\quad Wendao Wu$^{1,*}$\quad Chenheng Zhang$^{1,*}$\\[0.20em]
Yanyi Li$^1$\quad Chunyuan Zheng$^1$\\[0.20em]
Cong Fang$^{1,\dagger}$\quad Haoxuan Li$^{1,\dagger}$\quad Zhouchen Lin$^{1,\dagger}$\\[0.60em]
\small $^1$\affiliation\\[0.35em]
\footnotesize\mail{\emailA} \\[0.35em]
\small $^*$Equal contribution.\quad $^\dagger$Corresponding authors.}
\hypersetup{pdfauthor={\authornames}}
\fi
\date{}
\newcommand{\ResearchAgentSystem}{our laboratory's internal auto-research system}

\begin{document}
\maketitle

\begin{abstract}
We establish a randomized first-order lower bound for finding an $\eps$-stationary point, $\|\nabla f(x)\|\le\eps$, of a nonconvex function with initial gap at most $\Delta$, $L_1$-Lipschitz gradient, and $L_2$-Lipschitz Hessian. Each oracle call returns the exact function value and gradient. Let $\Q_{\rand,\FO}^{\infty}(\eps;\Delta,L_1,L_2)$ denote the minimax number of calls, maximized over finite dimensions, for arbitrary adaptive randomized algorithms with per-instance success probability at least $2/3$. In the regime $\eps\lesssim L_1^2/L_2$ and $\Delta L_2^{1/2}\eps^{-3/2}\gtrsim1$, we prove
\[
 \Q_{\rand,\FO}^{\infty}(\eps;\Delta,L_1,L_2)
 \ge c\,\Delta L_1^{1/2}L_2^{1/4}\eps^{-7/4},
\]
where $c>0$ is an absolute constant. This matches the deterministic restarted accelerated-gradient upper bound of Li and Lin \citep{LiLin2023JMLR} and extends the sharp deterministic lower bound of Zhou \citep{Zhou2026Sharp} to unrestricted randomized algorithms. Thus randomization does not improve the high-dimensional worst-case query rate. The construction keeps quadratic curvature visible while revealing the forcing directions sequentially. Large eigenspaces limit preprocessing, and a scheduled-disclosure coupling makes each fresh direction incur a separate cost.
\end{abstract}

\medskip
{\small
\begin{samepage}
\noindent\textbf{AI Usage.}
Nearly the entire research pipeline for this paper was carried out by Colombo,
\ResearchAgentSystem{}, powered by GPT-5.6 Sol. The system also conducted a
Lean-backed article audit of the resulting manuscript. The authors subsequently reviewed and approved the
mathematical claims, presentation, and formal artifacts, and take
responsibility for the final manuscript. The complete Lean audit report and the system's technical report will be made public at a later date.
\par
\end{samepage}
}

\paragraph{Publication Intent and Recognition of Independent Work.} 
This preprint is released as a research output of our study of
automated scientific research, for public scrutiny and discussion.
As we do not intend to submit this manuscript to a journal or
conference, it has not undergone the full editorial refinement
and internal review we would normally require for submission.
The authors remain responsible for its contents and for correcting
any errors or inaccurate attribution.

We recognize that releasing these results may affect researchers
who have independently been pursuing the same questions, including
work that is not yet public.
If this applies to your work, we invite you to contact us at
\texttt{colombo.2026@outlook.com}.
A brief description of the overlap, together with any dated
materials you are comfortable sharing, would help us understand
its scope and chronology.
Such materials might include draft excerpts, presentation slides,
or revision records; a complete manuscript or unpublished proof
is not required to begin the conversation.

We will review these cases carefully and update citations,
acknowledgments, and accounts of independent or concurrent work
as appropriate, including revising priority claims where warranted.
Nonpublic materials will be treated confidentially and used only
to assess the overlap unless their authors authorize otherwise.
We will obtain explicit permission from the authors before
publicly discussing such materials or providing them to an
automated research system.

\section{Introduction}

We study the number of exact function-value and gradient queries needed to find a point $x$ with $\|\nabla f(x)\|\le\eps$ for a smooth nonconvex function $f:\R^d\to\R$. The function has initial gap $f(0)-\inf f\le\Delta$, an $L_1$-Lipschitz gradient, and an $L_2$-Lipschitz Hessian, where $\Delta,L_1,L_2>0$. Hessian regularity constrains the objective, but the algorithm receives only first-order information. Our question is whether randomization reduces the worst-case number of queries in this setting.

With only a Lipschitz gradient, the classical gradient-descent rate is $O(\Delta L_1\eps^{-2})$, and its accuracy exponent is optimal in the unrestricted oracle model \citep{CDHS2020I}. Additional Hessian regularity allows acceleration without Hessian queries. The ``convex until proven guilty'' method obtains the $\eps^{-7/4}$ accuracy dependence with a logarithmic factor \citep{Carmon2017Convex}. Li and Lin remove that factor with restarted accelerated gradient \citep{LiLin2022Restarted,LiLin2023JMLR}, giving
\begin{equation}\label{eq:upper-intro}
 O\!\left(\Delta L_1^{1/2}L_2^{1/4}\eps^{-7/4}\right)
\end{equation}
first-order calls in the high-accuracy regime. This upper bound is deterministic and is therefore available to randomized algorithms as well.

The corresponding lower bounds have distinguished deterministic from randomized algorithms. Carmon et al.\ establish a deterministic $\Omega(\eps^{-12/7})$ lower bound under gradient and Hessian Lipschitzness \citep{CDHS2021II}. Zhou strengthens it to the sharp $\Omega(\eps^{-7/4})$ rate \citep{Zhou2026Sharp}. That deterministic conclusion does not exclude a faster randomized policy: an algorithm-dependent resisting rotation does not supply a hard distribution independent of the policy's private random seed.

We establish the matching randomized lower bound, with the full parameter dependence in \eqref{eq:upper-intro}. The new contribution is the necessity of this cost for arbitrary randomized first-order queries; the upper rate is inherited from Li and Lin. The hard distribution is fixed before the algorithm runs and consists entirely of functions in the prescribed smoothness class. It forces many quadratic systems to be solved sequentially, even though each query can explore every block matrix at once. In particular, every randomized method with per-instance success probability at least $2/3$ requires $\Omega(\Delta L_1^{1/2}L_2^{1/4}\eps^{-7/4})$ calls in sufficiently large dimension, when $\eps\lesssim L_1^2/L_2$ and $\Delta L_2^{1/2}\eps^{-3/2}\gtrsim1$. \Cref{tab:related} summarizes this extension; \Cref{thm:main} gives the formal statement.

\begin{table}[!htbp]
\centering
\small
\setlength{\belowcaptionskip}{8pt}
\caption{First-order query bounds for Hessian-Lipschitz nonconvex stationarity. Our contribution extends Zhou's deterministic lower bound to unrestricted randomized algorithms at the same rate, matching the existing Li--Lin upper bound. Rates fix $\Delta,L_1,L_2$ and take $\eps$ sufficiently small; lower bounds allow sufficiently large finite dimension, with randomized success probability at least $2/3$ on every instance. \Cref{thm:main} gives the full parameter dependence and regime.}
\label{tab:related}
\renewcommand{\arraystretch}{1.18}
\setlength{\tabcolsep}{4pt}
\begin{tabularx}{\linewidth}{@{}>{\raggedright\arraybackslash}X c c >{\raggedright\arraybackslash}p{0.21\linewidth}@{}}
\toprule
Work & Upper bound & Lower bound & Algorithms \\
\midrule
\multicolumn{4}{@{}l}{\textit{Established upper bounds}} \\
Carmon et al.\ \citep{Carmon2017Convex} & $O(\eps^{-7/4}\log(1/\eps))$ & --- & Deterministic \\
Li--Lin \citep{LiLin2023JMLR} & $O(\eps^{-7/4})$ & --- & Deterministic \\
\midrule
\multicolumn{4}{@{}l}{\textit{Lower bounds: from deterministic to randomized algorithms}} \\
Carmon et al.\ \citep{CDHS2021II} & --- & $\Omega(\eps^{-12/7})$ & Deterministic \\
Zhou \citep{Zhou2026Sharp} & --- & $\Omega(\eps^{-7/4})$ & Deterministic \\
\textbf{This Paper} (\Cref{thm:main}) & --- & $\Omega(\eps^{-7/4})$ & \textbf{Unrestricted randomized} \\
\bottomrule
\end{tabularx}
\end{table}

\paragraph{Proof idea.}
The construction leaves every quadratic block matrix accessible from the start but reveals the forcing directions sequentially. Large eigenspaces ensure that preprocessing a future matrix cannot substitute for learning its fresh forcing direction. A scheduled oracle reveals each forcing direction early and later discloses its full rotation; a coupling proves that even this additional information cannot complete the cascade quickly. Positive forcing makes completion necessary for stationarity. These ingredients make the per-stage information costs accumulate despite simultaneous access to all block matrices. The formal model and lower bound appear in \Cref{sec:model-results}; \Cref{sec:proof-strategy} explains the construction and its parameter scaling.

\section{Related Work}

\paragraph{Acceleration with first-order information.}
Higher-order smoothness can improve first-order query complexity without changing the oracle. Carmon et al.\ use accelerated gradient steps together with a certificate of nonconvexity to obtain a logarithmically modified $\eps^{-7/4}$ rate; under Lipschitz third derivatives, their method also gives a $\eps^{-5/3}$ rate up to logarithms \citep{Carmon2017Convex}. Li and Lin's restarted AGD and heavy-ball schemes obtain the Hessian-Lipschitz $\eps^{-7/4}$ rate without a polylogarithmic loss \citep{LiLin2022Restarted,LiLin2023JMLR}. Their restart analysis supplies the upper benchmark here. We do not propose a faster method; we study whether a randomized policy can improve on that benchmark.

\paragraph{Nonconvex first-order lower bounds.}
The zero-chain construction of Carmon et al.\ establishes a deterministic $\eps^{-12/7}$ lower bound with Lipschitz gradients and Hessians \citep{CDHS2021II}. Zhou's block-chain construction attains the deterministic $\eps^{-7/4}$ rate and treats higher finite smoothness orders \citep{Zhou2026Sharp}. The present construction addresses the additional requirement of randomization: one distribution must be hard for every deterministic policy, rather than selecting a rotation after fixing the policy. This requires controlling arbitrary preprocessing of future blocks and the vector information carried by gate gradients after activation.

\paragraph{Stronger derivative access.}
Hessian regularity and Hessian access lead to different complexity questions. The all-derivative lower bound of Carmon et al.\ has the Hessian-Lipschitz accuracy exponent $3/2$ \citep{CDHS2020I}. Since that oracle supplies more information, its lower bound also constrains first-order algorithms, but it does not establish the sharper first-order exponent $7/4$. Carmon et al.\ \citep{CarmonDuchiHinderSidford2018Accelerated} also analyze acceleration using gradient and Hessian-vector computations and obtain a second-order stationarity guarantee. Their Hessian-free terminology does not mean that a full Hessian is returned at each call; oracle accounting and the stronger output criterion must be specified when comparing that framework with the present exact value-gradient model.

\paragraph{Random rotations and linear-system complexity.}
Randomized quadratic lower bounds show how hidden spectral information constrains adaptive optimization \citep{Simchowitz2018Quadratic}. Derezi\'nski, Epperly, and Meyer establish matrix-vector lower bounds for linear systems \citep{DerezinskiEpperlyMeyer2026}; Shamir gives a direct random-rotation proof and a reduction to strongly convex quadratic optimization \citep{Shamir2026Simple}. Adaptive-to-Krylov reductions also appear in sampling lower bounds \citep{ChewiEtAl2023Sampling}. These results motivate the per-block information argument, but a single difficult system supplies only one stage of the desired nonconvex lower bound. Our scheduled reduction inserts a fresh right-hand side at a fixed round and discloses the whole block rotation after a prescribed delay. Large spectral multiplicities prevent earlier matrix queries from eliminating that delay; full disclosure makes subsequent gate gradients exactly computable.

\section{Problem Setup and Main Result}\label{sec:model-results}

We first specify the function class and the randomized query model. The main result extends the deterministic lower bound to arbitrary randomized policies under the same smoothness assumptions.

\subsection{Function class and oracle model}

For $d\in\mathbb N$, let
\begin{equation}\label{eq:function-class}
\cF_{1:2}^{(d)}(\Delta,L_1,L_2)
:=\left\{ f\in C^2(\R^d):
 f(0)-\inf_x f(x)\le \Delta,
 \Lip(\nabla f)\le L_1,
 \Lip(\nabla^2f)\le L_2
\right\}.
\end{equation}
All norms are Euclidean, with induced operator norms for derivatives. The exact first-order oracle returns $(f(x),\nabla f(x))$ at any $x\in\R^d$. An algorithm knows $d,\Delta,L_1,L_2,\eps$ and may choose arbitrary measurable queries using its private randomness and the complete transcript. Computation between calls is unrestricted, and the output need not be a queried point. For each $d$, let $\Q_{\rand,\FO}^{(d)}$ be the least almost-sure query cap under which some such algorithm outputs an $\eps$-stationary point with probability at least $2/3$ for every $f\in\cF_{1:2}^{(d)}$. Define $\Q_{\rand,\FO}^{\infty}:=\sup_{d\ge1}\Q_{\rand,\FO}^{(d)}$. Probability refers to the algorithm's private randomness on each fixed function.

\subsection{Randomized lower bound}

The following theorem states the new lower bound. The query cap is a worst-case cap, and the success requirement is imposed separately on every function in the class.

\begin{theorem}[Randomized first-order lower bound]\label{thm:main}
There are numerical constants $c,C>0$ such that, whenever
\begin{equation}\label{eq:regime}
  \eps\le c\frac{L_1^2}{L_2},
  \qquad
  \Delta L_2^{1/2}\eps^{-3/2}\ge C,
\end{equation}
one has
\begin{equation}\label{eq:main-lb}
  \Q_{\rand,\FO}^{\infty}(\eps;\Delta,L_1,L_2)
  \ge c\,\Delta L_1^{1/2}L_2^{1/4}\eps^{-7/4}.
\end{equation}
\end{theorem}

Combining \Cref{thm:main} with the deterministic restarted accelerated-gradient upper bound of \citet{LiLin2022Restarted,LiLin2023JMLR} yields
\begin{equation}\label{eq:main-theta}
  \Q_{\rand,\FO}^{\infty}(\eps;\Delta,L_1,L_2)
  =\Theta\!\left(\Delta L_1^{1/2}L_2^{1/4}\eps^{-7/4}\right)
\end{equation}
in the regime \eqref{eq:regime}. This matching characterization uses the existing upper bound; the new result is the randomized lower bound in \eqref{eq:main-lb}.

The lower bound is dimension-free in the usual oracle-complexity sense: for each horizon we choose a finite, sufficiently high dimension.  We do not optimize the dimension.  The fixed-dimensional frontier $\Q(d,\eps)$ is a different question.

\section{Proof Strategy and Parameter Calibration}\label{sec:proof-strategy}

\subsection{Why the costs multiply}

The proof of \Cref{thm:main} separates the number of stages from the work required after each new forcing direction becomes available:
\begin{equation}\label{eq:NM-intro}
\underbrace{N\asymp \Delta L_2^{1/2}\eps^{-3/2}}_{\text{sequential fresh right-hand sides}}
\times
\underbrace{M\asymp L_1^{1/2}L_2^{-1/4}\eps^{-1/4}}_{\text{post-revelation Krylov cost}}
=\Delta L_1^{1/2}L_2^{1/4}\eps^{-7/4}.
\end{equation}
The first factor is the length of a nonconvex cascade.  The second is the cost of approximately solving a positive-definite quadratic subproblem with condition number of order $L_1/(L_2\eps)^{1/2}$ once a fresh right-hand side is revealed.

A naive attempt to gate a full $L_1$-curvature quadratic block runs into a Hessian-Lipschitz barrier: switching curvature of size $L_1$ costs third derivative proportional to $L_1$ divided by the gate width.  Our construction instead keeps every quadratic matrix visible from the beginning and gates only a bounded positive forcing term.  The function has the form
\begin{equation}\label{eq:intro-F}
F(z_1,\ldots,z_N)
=\sum_{s=1}^N \frac12 z_s^\top A_s z_s
-\ell_1(z_1)-\sum_{s=2}^N \chi(v_{s-1}(z_{s-1}))\ell_s(z_s),
\end{equation}
where $\chi$ is flat on $(-\infty,1/4]$, $\ell_s$ is a positive bounded forcing, and $v_s$ is a hidden readout direction depending on the fresh right-hand side of block $s$.

There are two obstacles.  First, the algorithm may query all block matrices in parallel before later right-hand sides are revealed.  We defeat this using large eigenspace multiplicities.  In every eigenspace, the algorithm's entire preprocessing span has dimension at most the number of queries, whereas the subsequently revealed right-hand side has a fresh random direction.  A fixed dual residual polynomial supplies a readout vector almost orthogonal to all degree-$<M$ Krylov directions generated by that right-hand side.  \Cref{lem:highmult,cor:freshbarrier} turn this separation into a gate barrier for the first $M$ post-revelation rounds, up to an absolute constant.

Second, opening a gate may expose its physical readout direction through a gradient. A matrix-vector oracle alone need not be able to simulate that vector. We therefore compare the true transcript to an auxiliary oracle with a fixed disclosure calendar: block $s$ receives its right-hand side at round $1+(s-1)m$, and its entire rotation becomes public at round $1+sm$. Before full disclosure its outgoing coupling is censored; afterward the oracle computes that coupling and its gradient exactly. The high-multiplicity bound keeps every censored gate flat, so the auxiliary and true responses agree along the coupled run. \Cref{lem:marked,lem:firstexit} prove the information-span and coupling statements. Fixed times avoid a hidden-vector-dependent stopping-time assumption; disclosing each rotation after its hiding period handles all later vector feedback.

Finally, hiding information is useful only if stationarity requires that information. Positive forcing rules out a stationary point that stops partway through a gate. A two-projection elimination first places every saturation in its linear region; the remaining readout inequality then propagates activation from the first block to the last. \Cref{lem:stationarity} proves that an $\eps$-stationary point must open every gate and solve the final quadratic system. Together with the sequential reveal barrier, this gives the product cost in \eqref{eq:NM-intro}.

\subsection{Parameter calibration}

We prove the lower bound by Yao's principle.  Fix a horizon $T_0$ specified below. We choose a distribution independent of the algorithm and then analyze any deterministic policy against it. The distribution is over functions in \eqref{eq:function-class}; with probability at least $2/3$ over that distribution, the algorithm does not query an $\eps$-stationary point before $cT_0$ calls.  The lower bound for randomized algorithms follows.

Fix a sufficiently large numerical constant $H$ and a sufficiently small numerical constant $c_\gamma>0$.  Define
\begin{equation}\label{eq:scales}
 h:=H\eps,
 \qquad
 \gamma:=c_\gamma\sqrt{L_2h},
 \qquad
 R:=\frac{h}{\gamma},
 \qquad
 D:=hR=\frac{h^2}{\gamma}.
\end{equation}
Here $h$ is the right-hand-side norm scale, $\gamma$ is the bottom eigenvalue in each quadratic block, $R$ is the solution radius, and $D$ is the function-value cost per stage.  Let $\bar L:=L_1/4$ and choose
\begin{equation}\label{eq:MN}
 M:=\left\lfloor c_M\sqrt{\frac{L_1}{\gamma}}\right\rfloor,
 \qquad
 N:=\left\lfloor c_N\frac{\Delta}{D}\right\rfloor
 =\left\lfloor c_N\frac{\Delta\gamma}{h^2}\right\rfloor,
\end{equation}
where $c_M,c_N>0$ are sufficiently small numerical constants.  In the regime \eqref{eq:regime}, $M,N\ge 2$ after changing constants.  Moreover
\begin{align}
 NM
 &\asymp \frac{\Delta\gamma}{h^2}\sqrt{\frac{L_1}{\gamma}}
 =\Delta L_1^{1/2}\gamma^{1/2}h^{-2}
 \asymp \Delta L_1^{1/2}L_2^{1/4}\eps^{-7/4}.
 \label{eq:product}
\end{align}
We take $T_0:=4NM$ and use polynomial degree parameter $m:=\lfloor M/2\rfloor$ below. Thus $m$ is a fixed constant fraction of $M$; choosing $c_M$ sufficiently small ensures the inverse-approximation condition in \Cref{lem:markov}. The block dimension will be chosen after $T_0$ is fixed. To make the constant choices unambiguous, first fix the moment constants and a permissible upper bound on the weight error. Choose $c_\alpha,H,K_0,C_0,C_B$ as in \cref{lem:stationarity}; then choose the small constants $c_\gamma,c_M,c_N$ and the regime constants. Finally reduce $\delta,\eta$ to satisfy \eqref{eq:flat-margin} and choose the multiplicities. Reducing the weight error preserves all preceding moment estimates.

\section{The Hard Distribution}

The construction has two separate requirements. Each block must hide a useful readout from short Krylov searches, and the couplings must preserve global smoothness while forcing sequential progress. We first choose a spectral residual that witnesses the hidden readout, then use bounded forcing terms to join independent blocks.

\subsection{Smooth primitives}

We use three $C^\infty$ primitives with fixed-width transitions. Explicit formulas and derivative bounds are given in the appendix.

\begin{lemma}[Flat gate, saturation and radial squash]\label{lem:primitives}
There exist $C^\infty$ maps with the following properties.
\begin{enumerate}[label=(\roman*),leftmargin=2.5em]
\item A nondecreasing gate $\chi:\R\to[0,1]$ satisfying
\[
 \chi(t)=0\quad(t\le 1/4),
 \qquad
 \chi(t)=1\quad(t\ge 1/2),
 \qquad
 \chi'(t)\ge0,
\]
with $\{t:\chi'(t)\ne0\}\subset(1/4,1/2)$ and $\|\chi^{(k)}\|_\infty\le C_k$ for $1\le k\le4$.
\item For every numerical $K_0\ge1$, a saturation $\vartheta:\R\to\R$ satisfying
\[
 \vartheta(t)=t\quad(|t|\le K_0),
 \qquad
 |\vartheta(t)|\le K_0+1,
\]
with $\|\vartheta'\|_\infty\le1$ and $\|\vartheta^{(k)}\|_\infty\le C_k$ for $2\le k\le4$, where the constants do not depend on $K_0$.
\item For every $B>0$, a radial map $\mathcal R_B:\R^d\to\R^d$ of the form $\mathcal R_B(z)=r_B(\|z\|)z$ such that
\[
 \mathcal R_B(z)=z\quad(\|z\|\le B),
 \qquad
 \|\mathcal R_B(z)\|\le2B,
 \qquad
 \|D^k\mathcal R_B(z)\|_{\op}\le C_kB^{1-k}
\]
for $1\le k\le4$.
\end{enumerate}
\end{lemma}

The radial form is important: $\mathcal R_B(z)$ is a scalar multiple of $z$ and hence belongs to any linear information span containing $z$.

\subsection{A fixed dual residual}

The following polynomial witness is the source of the per-stage cost $M$.

\begin{lemma}[Uniform inverse-approximation lower bound]\label{lem:markov}
There are numerical constants $c_0,c_1>0$ such that, if $1\le m\le c_0\sqrt{\bar L/\gamma}$ and $2\gamma\le\bar L$, then
\[
 \inf_{p\in\mathcal P_{m-1}}
 \sup_{\lambda\in[\gamma,\bar L]}
 \left|\frac1\lambda-p(\lambda)\right|
 \ge \frac{c_1}{\gamma}.
\]
\end{lemma}

\begin{proof}
If the error were less than $1/(16\gamma)$, then $p(\gamma)\ge15/(16\gamma)$ and $p(2\gamma)\le9/(16\gamma)$, so $\|p'\|_\infty\ge3/(8\gamma^2)$.  Markov's inequality on $[\gamma,\bar L]$ gives
\[
 \|p'\|_\infty\le \frac{2(m-1)^2}{\bar L-\gamma}\|p\|_\infty
 \le C\frac{m^2}{\bar L\gamma},
\]
which is a contradiction for $m\le c_0\sqrt{\bar L/\gamma}$ with $c_0$ small enough.
\end{proof}

Minimax duality and finite-dimensional moment representation convert \Cref{lem:markov} into a finite spectral design.

\begin{lemma}[Finite dual spectral design]\label{lem:dual}
Under the hypotheses of \cref{lem:markov}, there are distinct points $\lambda_1,\ldots,\lambda_J\in[\gamma,\bar L]$ and positive weights $\nu_1,\ldots,\nu_J$ with $\sum_j\nu_j=1$ and $J\le2m+2$ such that
\begin{equation}\label{eq:dualL2}
 \inf_{p\in\mathcal P_{m-1}}
 \sum_{j=1}^J\nu_j\left(\frac1{\lambda_j}-p(\lambda_j)\right)^2
 \ge \frac{c_2}{\gamma^2}.
\end{equation}
\end{lemma}

Let $p_\star$ be the $L^2(\nu)$ projection of $\lambda^{-1}$ onto $\mathcal P_{m-1}$ and set
\begin{equation}\label{eq:resid}
 r(\lambda):=\lambda^{-1}-p_\star(\lambda).
\end{equation}
Then
\begin{equation}\label{eq:nuorth}
 \sum_j\nu_j r(\lambda_j)p(\lambda_j)=0\quad\forall p\in\mathcal P_{m-1},
 \qquad
 \|r\|_{L^2(\nu)}\ge c/\gamma.
\end{equation}

Choose large multiplicities $K_1,\ldots,K_J$ and define
\[
 d_0:=\sum_jK_j,
 \qquad
 \Lambda:=\diag(\lambda_1I_{K_1},\ldots,\lambda_JI_{K_J}).
\]
The multiplicities are chosen so that $K_j/d_0$ approximates $\nu_j$ and so that the high-multiplicity concentration estimates below hold.  Draw $U$ Haar on $O(d_0)$, set $A:=U\Lambda U^\top$, and draw $b$ independently and uniformly on the sphere of radius $h$.  Let $P_j$ be the spectral projector of $A$ at $\lambda_j$ and put $\omega_j:=\|P_jb\|^2/h^2$.  Define
\begin{equation}\label{eq:a-def}
 a:=\frac{r(A)b}{\|r(A)b\|}
\end{equation}
on the full-probability event $r(A)b\ne0$, and choose a fixed unit vector otherwise. Finally, set the deterministic gate scale
\begin{equation}\label{eq:alpha0}
 \alpha_0:=c_\alpha\frac h\gamma,
\end{equation}
where $c_\alpha>0$ is a sufficiently small numerical constant.

Let $\good_{\rm wt}$ be the event that $\omega_j$ is within a fixed small relative error of $\nu_j$ for every $j$.

\begin{lemma}[Moment stability and approximate short-Krylov orthogonality]\label{lem:blockmoments}
The multiplicities can be chosen so that $\Prob(\good_{\rm wt})\ge1-\zeta$ for any prescribed $\zeta>0$.  On $\good_{\rm wt}$, with
\[
 \alpha:=a^\top A^{-1}b,
 \qquad
 \beta:=a^\top A^{-1}a,
 \qquad
 \rho:=b^\top A^{-1}b,
\]
one has
\begin{equation}\label{eq:moment-bounds}
 c\frac h\gamma\le\alpha\le\frac h\gamma,
 \qquad
 c\frac1\gamma\le\beta\le\frac1\gamma,
 \qquad
 c\frac{h^2}\gamma\le\rho\le\frac{h^2}\gamma.
\end{equation}
Moreover, for every $p\in\mathcal P_{m-1}$,
\begin{equation}\label{eq:approx-orth}
 |a^\top p(A)b|\le C\delta\|p(A)b\|,
\end{equation}
where $\delta$ is the relative accuracy in $\good_{\rm wt}$.  If
\[
 s:=\frac{\alpha}{\alpha_0},
 \qquad
 k:=\frac{D\beta}{\alpha_0^2},
 \qquad
 r_0:=\frac{\rho}{D},
\]
then, for $c_\alpha$ sufficiently small,
\begin{equation}\label{eq:skr-main}
 1\le s\le C,
 \qquad
 c\le k\le C,
 \qquad
 c\le r_0\le1,
 \qquad
 kr_0\ge s^2.
\end{equation}
\end{lemma}

\begin{proof}
The weight vector $(\omega_j)$ has a Dirichlet law with parameters $(K_j/2)_j$, giving the weight event by chi-square concentration and a union bound.  On that event the empirical $\omega$-norm and the design $\nu$-norm are equivalent.  Since $r\perp\mathcal P_{m-1}$ in $L^2(\nu)$,
\[
 |\langle r,p\rangle_\omega|
 \le |\langle r,p\rangle_\omega-\langle r,p\rangle_\nu|
 \le C\delta\|r\|_\omega\|p\|_\omega,
\]
which is \eqref{eq:approx-orth}.  Also
\[
 \alpha=h\frac{\langle r,\lambda^{-1}\rangle_\omega}{\|r\|_\omega}.
\]
Under $\nu$, $\langle r,\lambda^{-1}\rangle_\nu=\|r\|_\nu^2$ by \eqref{eq:nuorth}. Relative weight accuracy also gives
\[
 |\langle r,\lambda^{-1}\rangle_\omega-\langle r,\lambda^{-1}\rangle_\nu|
 \le\delta\|r\|_\nu\|\lambda^{-1}\|_\nu
 \le\frac{\delta}{\gamma}\|r\|_\nu.
\]
Since $\|r\|_\nu\ge c/\gamma$, choosing $\delta$ below a fixed numerical threshold makes this error at most $\|r\|_\nu^2/2$. Norm equivalence now yields the positive lower bound on $\alpha$. Its upper bound follows from $A\succeq\gamma I$.  The upper bounds on $\beta$ and $\rho$ are immediate from $A\succeq\gamma I$.  Cauchy--Schwarz in the $A^{-1}$ inner product gives $\alpha^2\le\beta\rho$, which yields the lower bounds on $\beta$ and $\rho$.  The displayed bounds for $s,k,r_0$ follow from \eqref{eq:moment-bounds} and \eqref{eq:alpha0}.
\end{proof}

\subsection{Bounded positive cascade}

The spectral design supplies a hard direction for each quadratic block. We now couple the blocks without switching their quadratic curvature on and off. Bounded positive forcing makes the coupling derivatives small and later prevents cancellation at partially open gates.

Take $N$ independent copies $(A_s,b_s,a_s)$ of the block distribution.  Let
\begin{equation}\label{eq:vqell}
 B:=C_BR,
 \qquad
 v_s(z_s):=\frac{a_s^\top\mathcal R_B(z_s)}{\alpha_0},
 \qquad
 q_s(z_s):=\frac{b_s^\top z_s}{D},
\end{equation}
and
\begin{equation}\label{eq:ell}
 \ell_s(z_s):=D\bigl[C_0+\vartheta(q_s(z_s))\bigr],
 \qquad C_0>K_0+2.
\end{equation}
Then
\begin{equation}\label{eq:ellbounds}
 D\le\ell_s\le CD,
 \qquad
 \|\nabla\ell_s\|\le Ch,
 \qquad
 \|\nabla^2\ell_s\|\le C\gamma,
 \qquad
 \|D^3\ell_s\|\le C\frac{\gamma^2}{h}.
\end{equation}

\begin{construction}[Static hard instance]\label{con:cascade}
For $z=(z_1,\ldots,z_N)\in(\R^{d_0})^N$, define
\begin{equation}\label{eq:F}
 F(z)=\sum_{s=1}^N\frac12 z_s^\top A_sz_s
 -\ell_1(z_1)
 -\sum_{s=2}^N\chi\bigl(v_{s-1}(z_{s-1})\bigr)\ell_s(z_s).
\end{equation}
Write $c_1:=1$ and $c_s:=\chi(v_{s-1})$ for $s\ge2$.
\end{construction}

\section{Smoothness and Stationarity}

We verify admissibility for every realization before using any high-probability event. The stationarity argument then works on the spectral-weight event and shows why a successful output must reach the end of the cascade.

\subsection{Membership in the target class}

\begin{lemma}[One coupling derivative scales]\label{lem:coupling-smooth}
Let
\[
 G(x,y):=-\chi\!\left(\frac{a^\top\mathcal R_B(x)}{\alpha_0}\right)\ell(y),
\]
where $\alpha_0\asymp h/\gamma$, $B\asymp h/\gamma$, and \eqref{eq:ellbounds} holds.  Then
\[
 \|\nabla^2G\|_{\op}\le C\gamma,
 \qquad
 \|D^3G\|_{\op}\le C\frac{\gamma^2}{h}.
\]
\end{lemma}

\begin{proof}
The readout $v(x)=a^\top\mathcal R_B(x)/\alpha_0$ obeys
\[
 \|Dv\|\lesssim\gamma/h,
 \qquad
 \|D^2v\|\lesssim\gamma^2/h^2,
 \qquad
 \|D^3v\|\lesssim\gamma^3/h^3.
\]
Each second derivative is bounded by one of $D\gamma^2/h^2$, $h\gamma/h$, or $\gamma$.  Each third derivative is bounded by one of $D\gamma^3/h^3$, $h\gamma^2/h^2$, $\gamma\gamma/h$, or $\gamma^2/h$.  Since $D=h^2/\gamma$, the claim follows.
\end{proof}

\begin{lemma}[Bounded local overlap]\label{lem:overlap}
Suppose a function on a block product space is a sum of terms supported on one block or two adjacent blocks, with at most $q$ terms on each block and each edge. If each local Hessian and third derivative have operator norms at most $L$ and $M_3$, then the global norms are at most $CqL$ and $CqM_3$, independently of the number of blocks. Here $q$ is an absolute constant for the cascade.
\end{lemma}

\begin{proof}
Assign the at most $q$ terms on each support to $q$ layers, and color adjacent edges with two colors. Each layer and color has disjoint supports. For such a family with support projections $P_i$, the third derivative satisfies
\[
 \left|\sum_i D^3G_i[P_ix,P_iy,P_iz]\right|
 \le M_3\sum_i\|P_ix\|\|P_iy\|\|P_iz\|
 \le M_3\|x\|\|y\|\|z\|,
\]
by H\"older's inequality and $\|a\|_3\le\|a\|_2$ for finite sequences. The Hessian bound follows from Cauchy--Schwarz. Summing the at most $3q$ families proves both claims.
\end{proof}

\begin{proposition}[Function-class membership]\label{prop:smooth}
For sufficiently small $c_\gamma$ and $c_N$, every realization of $F$ satisfies
\[
 \Lip(\nabla F)\le L_1,
 \qquad
 \Lip(\nabla^2F)\le L_2,
 \qquad
 F(0)-\inf F\le\Delta.
\]
\end{proposition}

\begin{proof}
Because $\|A_s\|\le\bar L=L_1/4$, \cref{lem:coupling-smooth,lem:overlap} imply
\[
 \|\nabla^2F\|_{\op}\le \bar L+C\gamma\le L_1
\]
under \eqref{eq:regime} and for sufficiently small constants.  Likewise
\[
 \|D^3F\|_{\op}\le C\frac{\gamma^2}{h}=Cc_\gamma^2L_2\le L_2.
\]
All quadratic terms are nonnegative and each negative coupling has magnitude at most $CD$.  Hence $F(z)\ge-CND$ and $|F(0)|\le CD$, so $F(0)-\inf F\le CND\le\Delta$ by choosing $c_N$ small.
\end{proof}

\subsection{No partial-gate stationary points}

Flat gates alone would not suffice: the objective could admit stationary points inside a transition region. The next lemma rules this out by eliminating the outgoing forcing from two scalar projections of the same block gradient.

\begin{lemma}[Near-stationarity forces complete activation]\label{lem:stationarity}
Choose $c_\alpha$ small, then $H$ large, then $K_0,C_0$, and finally $C_B$ sufficiently large, as specified in the proof.  On the event $\good_{\rm wt}$ for every block, every $z$ satisfying $\|\nabla F(z)\|\le\eps$ obeys
\begin{equation}\label{eq:allgates}
 \chi(v_s(z_s))=1\quad(1\le s\le N-1),
 \qquad
 \|A_Nz_N-b_N\|\le\eps.
\end{equation}
Moreover, $\|z_s\|\le B$ for all $s$.
\end{lemma}

\begin{proof}
The moment bounds hold uniformly once the relative weight error is below a fixed threshold. Fix $c_\alpha$ so that $s_0=\alpha_s/\alpha_0\ge1$ on that event, and fix $H$ large enough for the projection errors below to be at most $1/8$. The elimination calculation below, when the radial map is the identity, gives a bound $|q_s|\le C_*$ independent of $K_0,C_0,C_B$. Choose $K_0>C_*$ and $C_0>K_0+2$. We next choose $C_B$ to ensure that the identity assumption used in that calculation actually holds.

Let $g_s:=\nabla_{z_s}F(z)$. The incoming forcing has norm at most $h$, and the outgoing contribution has norm at most $C(C_0+K_0+1)D/\alpha_0\le C'h$. Here $C'$ is fixed after $K_0,C_0,c_\alpha$ and is independent of $C_B$, because $\|D\mathcal R_B\|$ is uniformly bounded.  Hence
\[
 \gamma\|z_s\|
 \le \|A_sz_s\|
 \le \|g_s\|+C^{\prime}h
 \le \eps+Ch.
\]
Since $h=H\eps$, $\|z_s\|\le Ch/\gamma=CR$.  Taking $C_B$ larger gives $\|z_s\|\le B$, so $\mathcal R_B(z_s)=z_s$ and $D\mathcal R_B(z_s)=I$.

Fix $s$ and abbreviate
\[
 q=q_s(z_s),\quad v=v_s(z_s),\quad
 s_0=\frac{\alpha_s}{\alpha_0},\quad
 k_0=\frac{D\beta_s}{\alpha_0^2},\quad
 r_0=\frac{\rho_s}{D},
\]
where $\alpha_s=a_s^\top A_s^{-1}b_s$, $\beta_s=a_s^\top A_s^{-1}a_s$, and $\rho_s=b_s^\top A_s^{-1}b_s$.  For $s<N$,
\begin{equation}\label{eq:blockgrad}
 g_s=A_sz_s-c_s\vartheta'(q)b_s
 -\frac{D\phi(q_{s+1})\chi'(v)}{\alpha_0}a_s,
 \qquad \phi(t):=C_0+\vartheta(t)\ge1.
\end{equation}
Projecting in the $A_s^{-1}a_s$ and $A_s^{-1}b_s$ directions gives
\begin{align}
 \delta_s^a&:=\frac{a_s^\top A_s^{-1}g_s}{\alpha_0}
 =v-c_ss_0\vartheta'(q)-k_0\phi(q_{s+1})\chi'(v),
 \label{eq:proj-a-main}\\
 \delta_s^b&:=\frac{b_s^\top A_s^{-1}g_s}{D}
 =q-c_sr_0\vartheta'(q)-s_0\phi(q_{s+1})\chi'(v),
 \label{eq:proj-b-main}
\end{align}
with $|\delta_s^a|+|\delta_s^b|\le C\eps/h=C/H$.  For $s=N$ the terms involving $\phi\chi'$ are absent.

If $s=N$ or $\chi'(v)=0$, \eqref{eq:proj-b-main} and the boundedness of $r_0,\vartheta'$ imply $|q|\le C$.  Otherwise $v\in(1/4,1/2)$ and eliminating $\phi(q_{s+1})\chi'(v)$ gives
\begin{equation}\label{eq:eliminate}
 \frac{k_0}{s_0}q
 =v+c_s\left(\frac{k_0r_0}{s_0}-s_0\right)\vartheta'(q)
 +\frac{k_0}{s_0}\delta_s^b-\delta_s^a.
\end{equation}
By \eqref{eq:skr-main}, $k_0/s_0$ is bounded above and below by positive constants, $k_0r_0\ge s_0^2$, and the right-hand side is numerically bounded.  Thus $|q|\le C$ with a constant independent of $K_0,C_0$.  The choice $K_0>C_*$ made above covers this bound.  Then $\vartheta(q_s)=q_s$ and $\vartheta'(q_s)=1$ for every block.

Equation \eqref{eq:proj-a-main} becomes
\[
 v=c_ss_0+k_0\phi(q_{s+1})\chi'(v)+\delta_s^a
 \ge c_ss_0-|\delta_s^a|.
\]
The choices of $c_\alpha,H$ ensure $s_0\ge1$ and $|\delta_s^a|\le1/8$.  Starting with $c_1=1$ gives $v_1\ge7/8$, hence $\chi(v_1)=1$.  Induction gives $\chi(v_s)=1$ for all $s<N$.  In the last block, $c_N=1$ and the saturation is linear, so $g_N=A_Nz_N-b_N$.
\end{proof}

\section{Information Lower Bound}

The quadratic matrices are visible from the first query, so a sequential argument must survive arbitrary preprocessing. We give the algorithm additional information on a fixed calendar: each right-hand side is revealed for free, and its entire rotation is disclosed after $m$ more rounds. Before disclosure, a high-multiplicity argument keeps that block's gate flat. After disclosure, the oracle can compute its gate gradient exactly. This separates the period in which hiding is needed from the period in which the readout direction is used in a response.

\subsection{A scheduled oracle with full disclosure}

Write $A_s=U_s\Lambda U_s^\top$, $b_s=hU_su_s$, where the $U_s$ are independent Haar rotations and the $u_s$ are independent uniform unit vectors, also independent of all rotations. This has the block distribution specified in Section~5. Fix deterministic times
\[
 \tau_s:=1+(s-1)m,\qquad \sigma_s:=1+sm.
\]
Block $s$ is \emph{unrevealed} before $\tau_s$, \emph{active} for $\tau_s\le t<\sigma_s$, and \emph{retired} for $t\ge\sigma_s$. These words describe the information supplied by the auxiliary oracle, not the value of the true gate.

\begin{definition}[Scheduled disclosure oracle]\label{def:scheduled}
The spectrum $\Lambda$ and block partition are public. Before query $t$, disclose $(u_s,b_s)$ whenever $t=\tau_s$, and disclose $U_s$ whenever $t=\sigma_s$. A query $z_t=(z_{s,t})_{s=1}^N$ returns $(A_sz_{s,t},U_s^\top z_{s,t})$ for every block. It also returns $(F_t(z_t),\nabla F_t(z_t))$, where
\begin{equation}\label{eq:censored-function}
 F_t(z):=\sum_{s=1}^N\frac12z_s^\top A_sz_s-\ell_1(z_1)
 -\sum_{s=2}^{k(t)}\chi(v_{s-1}(z_{s-1}))\ell_s(z_s),
 \qquad k(t):=\min\{N,1+\lfloor(t-1)/m\rfloor\}.
\end{equation}
An empty sum is zero. All disclosures precede selection of the query at that round.
\end{definition}

The last response in this definition is computable from the matrix-vector responses and the disclosed data. Each included coupling has a retired predecessor, whose $U_{s-1}$ and $u_{s-1}$ determine the \emph{physical} vector $a_{s-1}$. Its successor's $b_s$ is already disclosed. Thus both the scalar gate and its vector derivative are available, including outside the identity region of $\mathcal R_B$. The quadratic value uses $z_s^\top(A_sz_s)/2$. No undisclosed $a_s$ is used. We will prove equality with the true first-order response on a high-probability event; no unconditional domination claim is needed for the censored response.

\subsection{The information span before retirement}

\begin{lemma}[Scheduled adaptive-to-Krylov reduction]\label{lem:marked}
Fix a deterministic policy for \cref{def:scheduled}, a horizon $T$, and $d_0>C(T+1)^2$. For each block $s$, there exist vectors $g_{s,1},\ldots,g_{s,T}$ jointly distributed as independent standard Gaussians and independent of $u_s$, such that, for $t\le T$ and $t<\sigma_s$,
\begin{equation}\label{eq:forest}
 U_s^\top z_{s,t}\in
 \spanop\{\Lambda^jg_{s,i}:1\le i\le t,\ 0\le j\le t-i\}
 +\1_{\{t\ge\tau_s\}}\spanop\{\Lambda^ju_s:0\le j\le t-\tau_s\}.
\end{equation}
The joint Gaussian and independence assertions also hold conditional on all raw draws of the other blocks. Independence between the Gaussian families for different blocks is not asserted or needed.
\end{lemma}

\begin{proof}
Fix $s$ and condition on $\{(U_i,u_i):i\ne s\}$. These draws are independent of $(U_s,u_s)$. Until $\sigma_s$, every response involving other blocks is a deterministic function of the queries and these fixed draws. By \eqref{eq:censored-function}, the current block contributes only its matrix products before $\tau_s$, and its matrix products and disclosed $b_s$ after $\tau_s$. Its outgoing gate is omitted until $\sigma_s$. Consequently the full query strategy, including all feedback from other blocks, is a deterministic adaptive strategy for the single-block oracle $(Ax,U^\top x)$ with a fixed-time disclosure $(u,hUu)$. This reduction conditions on raw independent draws, not on other blocks' adaptive transcripts.

Suppress $s$. Let $\mathscr H_{t-1}$ contain the single-block history and previously constructed Gaussian innovations, and let $\cL_{t-1}$ be the span of the latent vectors whose images under $U$ have been exposed: previous $\xi_i:=U^\top z_i$, $\Lambda\xi_i$, and $u$ if disclosed. Conditional on this information, the unused restriction
\[
 U:\cL_{t-1}^{\perp}\longrightarrow(U\cL_{t-1})^{\perp}
\]
is Haar. Before disclosure this statement remains true on additionally conditioning on $u$, since the history and the construction so far do not use $u$.

At the start of round $\tau_s$, $u$ is still an independent uniform vector. Disclosing $(u,hUu)$ fixes the image of its component outside $\cL_{t-1}$; adjoining $u$ to $\cL_{t-1}$ preserves conditional Haar randomness on the remaining complement. The disclosure time is deterministic, so no conditioning on an event defined by the hidden vector occurs.

Now decompose the chosen physical query as $z_t=z_t^{\parallel}+z_t^{\perp}$ relative to $U\cL_{t-1}$. If $z_t^{\perp}\ne0$, its latent unit direction
\[
 e_t:=U^\top z_t^{\perp}/\|z_t^{\perp}\|
\]
is conditionally uniform in $\cL_{t-1}^{\perp}$. Independently sample $a_t\sim N(0,P_{\cL_{t-1}})$ and a chi radius $R_t$ with $d_0-\dim\cL_{t-1}$ degrees of freedom, and set $g_t=a_t+R_te_t$. Its conditional law is $N(0,I)$, also conditional on $u$. If $z_t^{\perp}=0$, use a fresh independent standard Gaussian instead. Thus
\[
 \xi_t\in\cL_{t-1}+\spanop\{g_t\}.
\]
The response exposes the images of $\xi_t$ and $\Lambda\xi_t$. The innovation adds no information about the remaining action of $U$: in the nonzero case its only rotation-dependent part is $e_t$, whose image is the already fixed $z_t^{\perp}/\|z_t^{\perp}\|$; its other randomness is independent. In the zero case the whole innovation is independent. The conditional Haar invariant therefore holds after adjoining $\xi_t,\Lambda\xi_t$.

Induction shows that a start introduced at round $i$ has degree at most $t-i$ in query $t$, and the marked start $u$ has degree at most $t-\tau_s$. The conditional Gaussian laws prove joint independence of the innovations and independence from $u$. After retirement, pad the family to length $T$ with independent Gaussians; no span assertion is made then. This proves the lemma. The underlying conditional-rotation induction is the one used for matrix-vector lower bounds \citep{Shamir2026Simple}; the explicit schedule and retirement specify the additional information allowed here.
\end{proof}

\subsection{High multiplicity defeats preprocessing}

All spaces in this subsection are in latent coordinates. For one block, suppress $s$ and set
\begin{equation}\label{eq:spaces}
 \cV:=\operatorname{cyc}_{\Lambda}(\spanop\{g_1,\ldots,g_{T_0}\}),\qquad
 \cK_m:=\spanop\{u,\Lambda u,\ldots,\Lambda^{m-1}u\},\qquad
 \widetilde a:=\frac{r(\Lambda)u}{\|r(\Lambda)u\|}.
\end{equation}
Let $\Pi_j$ be the fixed projector onto the $\lambda_j$ eigenspace of $\Lambda$, and define
\[
 W_j:=\spanop\{\Pi_jg_1,\ldots,\Pi_jg_{T_0}\},\qquad
 u_j:=\Pi_ju/\|\Pi_ju\|.
\]
The zero denominators occur with probability zero. Conditional on the weights $\|\Pi_ju\|^2$ and the Gaussian starts, the $u_j$ are independent uniform directions in their respective eigenspaces. This follows from \cref{lem:marked}; it does not condition on the physical matrix $A$.

\begin{lemma}[Preprocessing-resistant projection bound]\label{lem:highmult}
Fix small numerical $\eta,\delta>0$ and a deterministic scheduled policy. The multiplicities can be chosen independently of the policy so that, with probability at least $1-\zeta$, simultaneously for all blocks:
\begin{enumerate}[label=(\roman*),leftmargin=2.5em]
\item the weight event $\good_{\rm wt}$ holds;
\item $\max_j\|P_{W_j}u_j\|\le\eta$;
\item every $w\in\cK_m+\cV$ satisfies
\begin{equation}\label{eq:corr}
 |\widetilde a^\top w|\le C(\delta+\eta)\|w\|.
\end{equation}
\end{enumerate}
It suffices to take
\[
 K_j\ge C[T_0^2+\eta^{-2}T_0+\delta^{-2}]\log(2JN/\zeta),\qquad
 |K_j/d_0-\nu_j|\le(\delta/4)\nu_j.
\]
\end{lemma}

\begin{proof}
Since all $\nu_j$ are positive, simultaneous relative approximation and the stated minimum multiplicities follow by taking a sufficiently large common scale and rounding its multiples of $\nu_j$. Normalized independent Gaussian block norms give the weight event by chi-square concentration, with total failure probability at most $\zeta/2$.

Conditional on the starts and weights in latent coordinates, $W_j$ has dimension at most $T_0$ and $u_j$ is uniform on a $K_j$-dimensional sphere. The squared projection has the beta law with parameters $\dim(W_j)/2$ and $(K_j-\dim(W_j))/2$ (with the zero-dimensional case trivial). The chi-square representation gives
\[
 \Prob\{\|P_{W_j}u_j\|>\eta\}\le2e^{-cK_j\eta^2}
 \quad\text{if }K_j\ge C\eta^{-2}T_0.
\]
A union bound over eigenspaces and blocks costs at most $\zeta/2$. The conditional version of \cref{lem:marked} supplies this estimate for each block regardless of the other draws; independence of the block events is unnecessary. The $T_0^2$ slack also ensures enough unused dimension in that lemma.

Set $\widehat\cV:=\bigoplus_jW_j$, which contains $\cV$. For $x\in\operatorname{cyc}_{\Lambda}(u)$, write $\Pi_jx=d_ju_j$. Then
\[
 \|P_{\widehat\cV}x\|^2=\sum_jd_j^2\|P_{W_j}u_j\|^2\le\eta^2\|x\|^2.
\]
In particular this holds for $x=\widetilde a$ and for every $x\in\cK_m$. Rotating \eqref{eq:approx-orth} into latent coordinates gives $|\widetilde a^\top k|\le C\delta\|k\|$ for $k\in\cK_m$. If $w=k+v$ with $v\in\cV$, then $|\langle k,v\rangle|\le\eta\|k\|\|v\|$, so
\[
 \|w\|^2\ge(1-\eta)(\|k\|^2+\|v\|^2),\qquad
 |\widetilde a^\top w|\le C\delta\|k\|+\eta\|v\|.
\]
These inequalities imply \eqref{eq:corr}.
\end{proof}

After fixing the geometric constants, choose $\delta,\eta$ sufficiently small that
\begin{equation}\label{eq:flat-margin}
 2BC(\delta+\eta)\le\alpha_0/4.
\end{equation}
The ratio $B/\alpha_0$ is a fixed numerical constant, so these choices are independent of the target problem parameters.

\begin{corollary}[Readout barrier before retirement]\label{cor:freshbarrier}
On the event of \cref{lem:highmult}, every scheduled query at a round $t<\sigma_s$ satisfies
\[
 |v_s(z_{s,t})|=\left|\frac{a_s^\top\mathcal R_B(z_{s,t})}{\alpha_0}\right|\le\frac14.
\]
Hence the gate and all its derivatives vanish at that query.
\end{corollary}

\begin{proof}
Before $\tau_s$, \cref{lem:marked} places $U_s^\top z_{s,t}$ in $\cV$; for $\tau_s\le t<\sigma_s$ it lies in $\cK_m+\cV$, since $t-\tau_s\le m-1$. The radial map commutes with rotations and preserves the direction of its argument. Its norm is at most $2B$, so \cref{eq:corr,eq:flat-margin} prove the bound. Flatness includes the endpoint $1/4$ by smoothness.
\end{proof}

\subsection{Coupling to the true transcript}

Run any deterministic first-order policy against the scheduled oracle, feeding it only the censored value and gradient. The auxiliary responses and disclosures are available to the analysis but may be ignored by the policy. Construct the Gaussian families and the event of \cref{lem:highmult} along this entire scheduled run before comparing it to the true run.

\begin{lemma}[First-exit identity]\label{lem:firstexit}
On that event, the true and censored value-gradient transcripts of the policy coincide through $T_0$.
\end{lemma}

\begin{proof}
At round $t$, the omitted terms in \eqref{eq:censored-function} have predecessors $s\ge k(t)$ with $s<N$. Each such block has $t<\sigma_s$. By \cref{cor:freshbarrier}, its gate and first derivative vanish at the scheduled query, so the omitted term contributes neither value nor any block gradient. Thus the true and censored responses agree there. Initially the policies have the same history. Inductively, equality of past responses gives equality of the next queries and hence equality of the next responses. This proves the assertion without conditioning on a hidden-vector-dependent reveal event.
\end{proof}

\begin{lemma}[Sequential activation and final solve]\label{lem:sequential}
On the simultaneous good event, block $s$ cannot have a nonzero gate at any queried point before $\sigma_s=1+sm$. No query at $t<\sigma_N=1+Nm$ is $\eps$-stationary.
\end{lemma}

\begin{proof}
The gate claim follows from \cref{lem:firstexit,cor:freshbarrier}. If $t<\tau_N=\sigma_{N-1}$, gate $N-1$ is flat, contrary to \cref{lem:stationarity}. For $\tau_N\le t<\sigma_N$, stationarity would require $\|A_Nz_{N,t}-b_N\|\le\eps$. This implies
\[
 a_N^\top z_{N,t}
 \ge\alpha_N-\eps/\gamma\ge3\alpha_0/4,
 \qquad
 \|z_{N,t}\|\le(h+\eps)/\gamma<B,
\]
after the fixed choices of $c_\alpha,H,C_B$. Thus $\mathcal R_B$ is the identity and $v_N(z_{N,t})\ge3/4$, contradicting the readout barrier for the still active last block. All inequalities apply to the scheduled run and transfer to the true run by \cref{lem:firstexit}.
\end{proof}

\section{Proof of the Main Theorem}

The preceding results now connect three statements: every sampled function is admissible, stationarity requires completion of the cascade, and each stage incurs a fresh information cost. The distributional bound below makes the final quantifier order explicit.

\begin{proposition}[Distributional lower bound]\label{prop:distlower}
There is an absolute constant $c>0$ such that every deterministic exact first-order algorithm using fewer than $cNM$ calls fails, with probability at least $2/3$ over the hard distribution, to query an $\eps$-stationary point.
\end{proposition}

\begin{proof}
Choose the per-block failure probabilities in \cref{lem:blockmoments,lem:highmult} so that all good events hold simultaneously with probability at least $5/6$.  The multiplicities depend only on $N,M,T_0$ and the prescribed failure probabilities. The concentration estimates are uniform over deterministic policies, so the hard distribution is fixed independently of the policy and of any private random seed.

On the good event, \cref{lem:sequential} excludes stationarity at every query with index at most $Nm$.  Since $m$ is a fixed constant fraction of $M$, this is at least $cNM$.  Conversely, \cref{lem:stationarity} says that every $\eps$-stationary point must have all first $N-1$ gates fully open and must solve the final system to residual at most $\eps$.  Hence no queried point is $\eps$-stationary before $cNM$ calls.  An arbitrary unqueried terminal output can be appended as one final query, changing constants only.
\end{proof}

\begin{proof}[Proof of \cref{thm:main}]
The hard distribution is supported on $\cF_{1:2}^{(d)}(\Delta,L_1,L_2)$ by \cref{prop:smooth}.  By \cref{prop:distlower}, every deterministic algorithm with fewer than $cNM$ first-order calls succeeds with probability at most $1/3$ under this distribution.  Fixing each private seed gives a deterministic policy. Integrating its success probability over the seed leaves the distributional success probability at most $1/3$, contradicting a per-instance guarantee of at least $2/3$. This is the distributional direction of Yao's principle; the good event may depend on the policy, while the hard distribution and its probability bound do not.  The product estimate \eqref{eq:product} gives the claimed lower bound.  For the upper bound, apply Theorem~1 of \citet{LiLin2023JMLR} with $L=L_1$, $\rho=L_2$, and its accuracy parameter equal to $\eps/82$, since its output gradient norm is at most $82$ times that parameter. Its momentum condition is satisfied when $\eps\le cL_1^2/L_2$. Integer epoch lengths and the final epoch add at most $O(L_1^{1/2}L_2^{-1/4}\eps^{-1/4})$ calls, absorbed by the main bound when $\Delta L_2^{1/2}\eps^{-3/2}\ge C$. The descent summation only uses a finite lower bound on $f$, so replacing its notation $\min f$ by $\inf f$ requires no minimizer assumption. Each gradient evaluation is available from our oracle. This proves \eqref{eq:main-theta} with the existing deterministic upper bound.
\end{proof}

\section{Discussion and Limitations}

The result is a high-dimensional worst-case oracle lower bound.  The block dimension is deliberately large: multiplicities must satisfy the horizon-dependent concentration bounds and approximate the spectral design weights; optimizing this dependence is outside the scope of the paper.  Dimension-dependent rates can be different.  In particular, low-dimensional upper bounds may beat the dimension-free $\eps^{-7/4}$ rate in certain regimes, and a full characterization of $\Q(d,\eps)$ remains open.

The proof is specialized to the Hessian-Lipschitz case.  Zhou's deterministic construction treats all finite higher-order smoothness orders \citep{Zhou2026Sharp}.  Extending the present randomized cascade to general finite $p$ would likely require a $p$th-order analogue of the positive forcing and readout design.

Finally, the lower bound is exact-oracle and deterministic-function worst case.  It does not address stochastic gradient noise, finite-sum structure, average smoothness assumptions, or additional problem geometry.

\bibliographystyle{unsrtnat}
\bibliography{references}

\appendix

\section{Additional Proof Details}

We give explicit smooth primitives and the finite spectral-design argument, then clarify exact simulation and probability accounting for the passage from the censored transcript to the original value-gradient oracle.

\subsection[Smooth primitives]{Construction for \Cref{lem:primitives}}
Let $e(t)=\exp(-1/t)$ for $t>0$ and $e(t)=0$ for $t\le0$, and define
\[
 S(t):=\frac{e(t)}{e(t)+e(1-t)},\qquad \chi(t):=S(4t-1).
\]
The denominator is positive, and $S$ is smooth, nondecreasing, zero on $(-\infty,0]$, and one on $[1,\infty)$. Thus $\chi$ has the stated thresholds and bounded derivatives. Set
\[
 \vartheta(t):=\int_0^t[1-S(|q|-K_0)]\,dq.
\]
The integrand is identically one near zero, so the absolute value creates no loss of smoothness. It lies between zero and one and vanishes for $|q|\ge K_0+1$. This gives all saturation properties, with derivative bounds independent of $K_0\ge1$.

For the radial map, put $\rho(r)=\int_0^r[1-S(q-1)]\,dq$ for $r\ge0$ and
\[
 \mathcal R_B(z)=
 \begin{cases}
 B\rho(\|z\|/B)z/\|z\|,&z\ne0,\\
 0,&z=0.
 \end{cases}
\]
It equals $z$ for $\|z\|\le B$ and has norm at most $2B$. Smoothness at zero follows from this identity. Away from zero, derivatives of the unit radial vector have norms bounded by numerical multiples of $\|z\|^{-k}$; on the transition annulus $B\le\|z\|\le2B$, differentiation and scaling give $C_kB^{1-k}$. Beyond the annulus $\rho$ is constant and the same bounds follow from the unit-vector derivatives. These bounds are independent of dimension for each fixed $k\le4$.

\subsection[Finite spectral design]{Proof of \Cref{lem:dual}}
For a probability measure $\mu$ on $[\gamma,\bar L]$ and $p\in\mathcal P_{m-1}$, define
\[
 \Phi(p,\mu)=\int(\lambda^{-1}-p(\lambda))^2\,d\mu(\lambda).
\]
The space of probability measures on $[\gamma,\bar L]$ is compact and convex in the weak topology. For fixed $p$, $\Phi$ is continuous and affine in $\mu$; for fixed $\mu$, it is continuous and convex in the coefficients of $p$. Sion's minimax theorem, with the measure space as the compact side \citep{Sion1958Minimax}, therefore gives
\[
 \sup_\mu\inf_{p\in\mathcal P_{m-1}}\Phi(p,\mu)
 =\inf_{p\in\mathcal P_{m-1}}\sup_\mu\Phi(p,\mu)
 =\inf_p\|\lambda^{-1}-p\|_\infty^2\ge c/\gamma^2.
\]
No compact restriction of the polynomial coefficients is imposed. The infimum of the continuous functions $\Phi(p,\cdot)$ is upper semicontinuous, so a maximizing measure $\mu_*$ exists.

For completeness, finite support follows directly from convex geometry. The squared error depends only on the $2m+1$ coordinates of the continuous moment map
\[
 \psi(\lambda)=(\lambda^{-2},\lambda^{-1},1,\lambda,\ldots,\lambda^{2m-2}),
\]
where for $m=1$ the list stops at $1$. Its integral under $\mu_*$ belongs to the convex hull of the compact image of $\psi$. Carath\'eodory's theorem represents that integral as a convex combination of at most $2m+2$ image points. Delete zero weights and merge repeated points. The resulting positive atomic probability measure preserves $\Phi(p,\mu_*)$ for every $p$, proving \eqref{eq:dualL2}. This is the finite-function moment argument underlying Tchakaloff-type formulas \citep{BayerTeichmann2006}; it applies to the negative powers here without treating them as ordinary polynomial moments.

\subsection{Exact simulation after retirement}

The calendar oracle uses each block's matrix-vector responses from the start, its physical $b_s$ from $\tau_s$, and its physical $a_s$ only from $\sigma_s$. At retirement, the disclosed $U_s$ gives
\[
 a_s=U_s\frac{r(\Lambda)u_s}{\|r(\Lambda)u_s\|},\qquad
 \nabla v_s(z_s)=\frac{D\mathcal R_B(z_s)^\top a_s}{\alpha_0}.
\]
Thus every retained coupling in \eqref{eq:censored-function} can be differentiated exactly. Before retirement no attempt is made to recover $a_s$ from a scalar latent inner product; that outgoing coupling is omitted. \Cref{cor:freshbarrier,lem:firstexit} then prove that the omission changes neither value nor gradient along the coupled run on the good event.

\subsection{Why no union bound over adaptive queries is needed}

The adaptive algorithm may choose uncountably many possible queries as a function of its transcript.  The proof does not union bound over these queries.  Instead, \cref{lem:marked} embeds each block's queries before retirement into a random finite-dimensional latent span generated by $T_0$ Gaussian starts and its marked RHS branch. No span restriction is needed after full disclosure.  After this reduction, \cref{lem:highmult} controls every vector in the realized random subspace $\cK_m+\cV$ by a finite collection of eigenspace projection events.  This is why the multiplicity condition depends on $T_0,J,N$ but not on an $\eps$-net over query space.

\subsection{Good-event probability}

Apply \cref{lem:highmult} with $\zeta=1/6$, allocating at most $1/(12N)$ per block to each of the weight and projection events. A union bound gives simultaneous validity with probability at least $5/6$.  The remaining $1/6$ slack is absorbed into the constant-probability lower bound in \cref{prop:distlower}.  The hard distribution itself is unchanged on the bad event and remains supported on the target function class by \cref{prop:smooth}; only the certificate proving failure is lost.

\end{document}